\documentclass{amsart}

\usepackage{amssymb}
\usepackage{mathtools}
\usepackage{url}
\usepackage{hyperref}
\usepackage{color}
\usepackage{bm}
\usepackage{ulem}

\newtheorem{theorem}{Theorem}[section]
\newtheorem{lemma}[theorem]{Lemma}

\newtheorem{alg}[theorem]{Algorithm}

\theoremstyle{definition}
\newtheorem{definition}[theorem]{Definition}
\newtheorem{assumption}[theorem]{Assumption}

\theoremstyle{remark}
\newtheorem{remark}[theorem]{Remark}

\numberwithin{equation}{section}

\newcommand{\N}{\mathbb{N}}
\newcommand{\Z}{\mathbb{Z}}
\newcommand{\Q}{\mathbb{Q}}

\newcommand{\Norm}{\mathrm{N}}
\newcommand{\trace}{\mathrm{Tr}}
\newcommand{\order}{\mathcal{O}}
\newcommand{\pe}{\mathfrak{p}}
\newcommand{\il}{\mathfrak{l}}
\newcommand{\ord}{\mbox{ord}}
\newcommand{\weber}{\bm{\gamma}_3(E)}

\begin{document}

\title[Primality proving using CM elliptic curves of class number three]
    {Primality proving using elliptic curves with complex multiplication by imaginary quadratic fields of class number three}


\author{Hiroshi Onuki}
\address{Department of Mathematical Informatics, 
    The University of  Tokyo,
    7-3-1 Hongo, Bunkyo-ku, Tokyo 113-8656, Japan
    }
\curraddr{}
\email{onuki@mist.i.u-tokyo.ac.jp}
\thanks{}

\subjclass[2020]{Primary 11Y11; Secondary 11G05, 14K22}

\date{}

\dedicatory{}

\begin{abstract}
    In 2015, Abatzoglou, Silverberg, Sutherland, and Wong
    presented a framework for primality proving algorithms for special sequences of integers
    using an elliptic curve with complex multiplication.
    They applied their framework to obtain algorithms for elliptic curves with
    complex multiplication by imaginary quadratic field of class numbers one and two,
    but, they were not able to obtain primality proving algorithms
    in cases of higher class number.
    In this paper, we present a method to apply their framework to imaginary quadratic fields of class number three.
    In particular,
    our method provides a more efficient primality proving algorithm
    for special sequences of integers
    than the existing algorithms
    by using an imaginary quadratic field of class number three in which 2 splits.
    As an application,
    we give two special sequences of integers derived from $\Q(\sqrt{-23})$ and $\Q(\sqrt{-31})$,
    which are all the imaginary quadratic fields of class number three in which 2 splits.
    Finally, we give a computational result for the primality of these sequences.
\end{abstract}

\maketitle

\section{Introduction}
Research on algorithms to determine the primality of a given integer has
received a lot of interest from computational mathematics and its application to cryptography.
There are efficient algorithms for compositeness test, which return the input is composite or undetermined.
For example,
the Miller-Rabin test \cite[Algorithm 3.5.1]{crandall2005prime} has computational complexity $\tilde{O}((\log n)^2)$ with input $n$.
However, primality proving algorithms, which determine the primality of the input, are less efficient.
The best algorithm currently known is the Agrawal-Kayal-Saxena algorithm \cite{AKSprime} improved by
Lenstra and Pomerance \cite{LPpirme2016}
and has computational complexity $\tilde{O}((\log n)^6)$ with input $n$.

On the other hand,
more efficient algorithms are known
for an integer $n$ such that a partial factorization of either
$n + 1$ or $n - 1$ is known (see \S 4.1 and 4.2 of \cite{crandall2005prime}).
For example,
the Lucas-Lehmer test for Mersenne numbers \cite[Theorem 4.2.6]{crandall2005prime}
and the Pepin test for Fermat numbers \cite[Theorem 4.1.2]{crandall2005prime}
are such algorithms.
The computational complexities of these algorithms are $\tilde{O}((\log n)^2)$
the same as for the Miller-Rabin test.
These classical algorithms (we call these the \textit{$n \pm 1$ algorithms})
use the structure of the multiplicative group of the residue ring modulo $n$.

In 2005, Gross \cite{Gross2005} constructed a primality proving algorithm for
Mersenne numbers by using an elliptic curve with complex multiplication (CM)
by $\Q(\sqrt{-1})$.
Its computational complexity is also $\tilde{O}((\log n)^2)$ with input $n$.
Similar results for other sequences of integers are obtained
by Denomme and Savin \cite{DS_prime2008} for Fermat numbers and integers of the
forms $3^{2^l} - 3^{2^{l-1}} + 1$ and $2^{2^l} - 2^{2^{l-1}} + 1$,
and by Tsumura \cite{Tsumura2011} for integers of the form $2^{p} \pm 2^{(p+1)/2} + 1$.
Denomme and Savin used elliptic curves with CM by $\Q(\sqrt{-1})$ and $\Q(\sqrt{-3})$,
and Tsumura used an elliptic curve with CM by $\Q(\sqrt{-1})$.
However, an integer $N$ to which these algorithms can be applied has
the property that $N + 1$ or $N - 1$ is smooth, i.e., 
the $n \pm 1$ algorithms can also be applied to such $N$.

Based on these studies,
Abatzoglou, Silverberg, Sutherland and Wong \cite{Abatzoglou_2013}
proposed a primality proving algorithm for sequences of integers
to which the $n \pm 1$ algorithms cannot be applied
(see also \cite{Wong_thesis}).
Their algorithm has the same computational complexity as the above algorithms.
The same authors \cite{Abatzoglou2015} presented
a framework for primality proving algorithms using elliptic curves with CM.
If one applies their framework to elliptic curves with CM by
an imaginary quadratic field in which $2$ splits,
i.e., $2$ splits into two distinct prime ideals,
then one can obtain a primality proving algorithm for sequences of integers
to which the $n \pm 1$ algorithms cannot be applied.
As an application of their framework,
they gave a primality proving algorithm for sequences of integers
using $\Q(\sqrt{-7})$ and $\Q(\sqrt{-15})$, which are 
all the imaginary quadratic fields of class number one or two
in which $2$ splits.
A primality proving algorithm using $\Q(\sqrt{-2})$
based on the same framework
was also proposed by \cite{Abatzoglou_thesis,Abatzoglou2015}.
However, the $n \pm 1$ algorithm can be applied to
the sequence of integers to which the algorithm using $\Q(\sqrt{-2})$ can be applied
since $2$ is ramified in $\Q(\sqrt{-2})$.
The algorithm using $\Q(\sqrt{-15})$ is the first algorithm of this type using an imaginary quadratic field of class number two
(the previous algorithms use those of class number one).
They left as an open problem
the construction of a similar primality proving algorithm using an imaginary quadratic field of class number greater than two
(see Remark 4.12 in \cite{Abatzoglou2015}).
The main obstacle to use an imaginary quadratic field of class number greater
than two is as follows.
Let $K$ be an imaginary quadratic field
and $H$ the Hilbert class field of $K$.
The framework of Abatzoglou et al. uses a sequence $\{\pe_k\}_k$ of ideals of $\order_H$
such that $\Norm_{H/K}(\pe_k)$ is principal 
and a generator $\pi_k$ of $\Norm_{H/K}(\pe_k)$ satisfying that
$\pi_k - 1$ is ``highly factored,''
i.e., $\pi_k - 1$ is divisible by a smooth $\alpha \in \order_K$
and $\Norm_{K/\Q}(\alpha)$ is approximately greater than $\sqrt{\Norm_{K/\Q}(\pi_k)}$,
and they did not know how to {\bf systematically} find such ideals.

In this paper, we solve this problem in the case that the class number is three.
Our method is based on the framework of Abatzoglou et al. and
can be applied to any imaginary quadratic field of class number three
in which $2$ splits.
As in the framework of Abatzoglou et al.,
we can obtain a primality proving algorithm for a special sequence of integers
to which the $n\pm 1$ algorithms cannot be applied.

Let $K$ be an imaginary quadratic field of class number three
and $H$ the Hilbert class field of $K$.
The key idea of this work is using a primitive element $\xi$ of $H/K$
such that $\trace_{H/K}(\xi) = 0$.
In this case, the minimal polynomial of $\xi$ over $K$ is
of the form
$x^3 + c_1x + c_0$ for $c_0, c_1 \in \order_K$.
We define a sequence $\{\pe_k\}_k$ of ideals of $\order_H$ as
$\pe_k = (1 - \alpha^k\xi)\order_H$ for a smooth $\alpha \in \order_K$.
Then $\Norm_{H/K}(\pe_k)$ is principal and generated by
$\pi_k = 1 + c_1\alpha^{2k} + c_0\alpha^{3k}$
and $\pi_k - 1 = \alpha^{2k}(c_1 + c_0\alpha^k)$ is highly factored.

As an application,
we present two sequences of integers and primality proving algorithms for them.
These algorithms use $\Q(\sqrt{-23})$ and $\Q(\sqrt{-31})$, respectively,
which are all the imaginary quadratic fields of class number three
in which $2$ splits
(see sequence A006203 in \cite{Sloane_enc}).

The organization of this paper is as follows.
In \S \ref{sec:notation}, we give the notation that we use in this paper.
We recall the framework of \cite{Abatzoglou2015} in \S \ref{sec:prev_work}.
Our main contribution is given in \S \ref{sec:new_method}.
In particular, we present our idea to solve the open problem in \S \ref{subsec:method_pek}.
Then, in \S \ref{subsec:our_setting}, we give the setting of our primality proving.
We deal with some technical issues to construct our algorithm in \S\ref{subsec:cond_k}.
Our concrete algorithm is given in \S\ref{subsec:algorithm}.
Finally, in \S\ref{sec:examples}, we show the computational result of our primality proving algorithm.

\section{Notation}\label{sec:notation}
Let $L$ be a number field.
We denote the ring of integers of $L$ by $\order_L$.
The norm and the trace of a finite extension $M/L$ are denoted by $\Norm_{M/L}$ and $\trace_{M/L}$, respectively.
For a place $v$ of $L$,
we denote the completion of $L$ at $v$ by $L_v$.
For a finite place $\pe$ of $L$ and $a \in L_\pe$,
we denote the $\pe$-adic valuation of $a$ by $\ord_\pe(a)$.
For $a, b \in L_v^\times$, the \textit{quadratic Hilbert symbol
$\left(\frac{\cdot,\cdot}{v}\right)$ at $v$} is defined by
\begin{equation*}
    \left(\frac{a,b}{v}\right) = \left\{
        \begin{array}{ll}
            1, & \mbox{if } ax^2 + by^2 = z^2 \mbox{ has a solution } (x, y, z) \in L_v^3\backslash(0,0,0),\\
            -1, & \mbox{otherwise}.
        \end{array}
    \right.
\end{equation*}
For a finite place $\pe$ of $L$ not dividing $2$,
we define the \textit{quadratic residue symbol modulo $\pe$} as
\begin{equation*}
    \left(\frac{a}{\pe}\right) = \left\{
        \begin{array}{cl}
            \left(\frac{\pi, a}{\pe}\right), & \mbox{if } \ord_\pe(a) \mbox{ is even},\\
            0, & \mbox{otherwise},
        \end{array}
    \right.
\end{equation*}
where $a \in L_\pe^\times$ and $\pi$ is a prime element of $L_\pe$.
Note that this definition does not depend on the choice of $\pi$.
If $\mathrm{ord}_\pe(a) = 0$ then $\left(\frac{a}{\pe}\right) = 1$ if and only if $a$ is square modulo $\pe$.
For an element $a$ of an imaginary quadratic field,
$\bar{a}$ denotes the complex conjugate of $a$.

In this paper, we consider only elliptic curves defined by a short Weierstrass form,
i.e., curves defined by a projective equation
$Y^2Z = X^3 + AXZ^2 + BZ^3$,
where $A$ and $B$ are the coefficients of the curve.
For simplicity, we also use the affine model $y^2 = x^3 + Ax + B$ of the same curve. 
For an elliptic curve $E$,
we denote its neutral element by $O_E$,
its $j$-invariant by $j(E)$,
its discriminant by $\mathrm{disc}(E)$,
and its endomorphism ring by $\mathrm{End}(E)$.
Note that $O_E$ is the unique point whose $Z$-coordinate is zero.
Let $\order$ be an order of an imaginary quadratic field and
$E$ an elliptic curve over a number field.
When $\mathrm{End}(E)$ is isomorphic to $\order$,
we say $E$ is 
an \textit{elliptic curve with complex multiplication (CM) by $\order$}.
In this case, the normalized isomorphism
$\order \to \mathrm{End}(E)$ is denoted by $[\cdot]$
(see \cite[Proposition II.1.1]{silverman1994advanced} for the definition).

\section{Framework for primality proving}\label{sec:prev_work}
In this section,
we give an overview of the framework provided in \cite{Abatzoglou2015}.
First, we recall terminology from \cite{Abatzoglou2015}.
\begin{definition}
    Let $E$ be an elliptic curve over a number field $M$
    and $J$ an ideal of $\order_{M}$ prime to $\mathrm{disc}(E)$.
    We say that $P \in E(M)$ is {\it strongly nonzero modulo $J$}
    if one can express $P = (X:Y:Z) \in \mathbb{P}^2(\order_{M})$
    such that $ Z\order_{M} + J = \order_{M}$.
\end{definition}

Let $K$ be an imaginary quadratic field,
$E$ an elliptic curve defined over a number field $M$ containing $K$
with CM by $\order_{K}$.
We fix a point $P$ in $E(M)$.
Let $\gamma, \alpha_1, \dots, \alpha_s$ be nonzero elements of $\order_{K}$
and assume that we know the prime ideal factorization of $(\prod_{i=1}^s\alpha_i)\order_K$.
For $k = (k_1, \dots, k_s) \in \N^{s}$,
we define
$$
\Lambda_k = \gamma\alpha_1^{k_1}\cdots\alpha_s^{k_s},
\quad
\pi_k = 1 + \Lambda_k,
\quad
F_k = \Norm_{K/\Q}(\pi_k),
$$
and $L_k$ as the positive generator of an ideal $(\Lambda_k/\gamma) \order_K \cap \Z$ of $\Z$.
Let $\pe_k$ be an ideal of $\order_{M}$ such that
$\Norm_{M/K}(\pe_k) = \pi_k\order_{K}$.
When $\pe_k$ is a prime ideal,
we denote the reductions modulo $\pe_k$ of $E$ and $P$
by $\tilde{E}$ and $\tilde{P}$.
Further, we adopt the assumption described in \S 3.1 of \cite{Abatzoglou2015}.
In particular, we assume that $k$ is chosen so that the following holds.
\begin{assumption}\label{assump:frob_inv}
  Whenever $\pe_k$ is a prime ideal, the following hold:
  \begin{enumerate}
    \renewcommand{\labelenumi}{(\roman{enumi})}
    \item The Frobenius endomorphism of $\tilde{E}$ is $[\pi_k]$,
    \item $\tilde{P} \not\in
          [\lambda] \tilde{E}(\order_{M}/\pe_k)$
          for every prime ideal $\lambda$ of $\order_{K}$ containing
          $\prod_{i=1}^s\alpha_i$.
  \end{enumerate}
\end{assumption}
In this setting,
Abatzoglou et al. showed
the following primality theorem.

\begin{theorem}[Theorem 3.6 of \cite{Abatzoglou2015}]\label{thm:abat2}
    With the notation as above, assume Assumption \ref{assump:frob_inv}.
    Suppose that $\prod_{i=1}^s\alpha_i$ is not divisible in $\order_{K}$
    by any rational prime that splits in $K$,
    and that $F_k > 16\Norm_{K/\Q}(\gamma\prod_{\lambda}\lambda)^2$
    where $\lambda$ runs over the prime ideals of $\order_{K}$
    that divide $\prod_{i=1}^s\alpha_i$ and are ramified in $K/\Q$.
    Then the following are equivalent:
    \begin{enumerate}
    \renewcommand{\labelenumi}{{\rm (\alph{enumi})}}
    \item $\pe_k$ is a prime ideal,
    \item $[L_k \gamma]P \equiv O_E \pmod{\pe_k}$,
        and $\left[\frac{L_k}{p}\gamma\right]P$
        is strongly nonzero modulo $\pe_k$
        for every prime divisor $p$ of $\Norm_{K/\Q}(\alpha_1\cdots\alpha_s)$.
  \end{enumerate}
\end{theorem}

This theorem states that the primality of $\pe_k$ can be
detected by the arithmetic of $P$ modulo $\pe_k$.
The condition $F_k > 16\Norm_{K/\Q}(\gamma\prod\lambda)^2$ means that
the unknown part of the factorization of $\pi_k - 1$ is less than
about half of the size of $\pi_k$.
This and condition (ii) in Assumption \ref{assump:frob_inv} ensure that
$P \bmod{\pe_k}$ has a sufficiently large order to prove the primality of $\pe_k$.
Although $\gamma$ does not depend on $k$ in \cite{Abatzoglou2015},
the proof does not use this fact.
Therefore, Theorem \ref{thm:abat2} remains true even if $\gamma$ depends on $k$.

Abatzoglou et al. obtained an algorithm for determining the primality of $\pe_k$ from this theorem.
To make the above theorem a concrete primality proving algorithm,
we must choose a setting that satisfies Assumption \ref{assump:frob_inv}.

As in \cite{Abatzoglou2015},
for condition (i),
we can use Theorem 5.3 of \cite{RS_pointCM2009}.
Suppose that $E$ is defined by $y^2 = x^3 + Ax + B$ and the discriminant of $K$ is odd.
Then the theorem says that
the Frobenius map of $\tilde{E}$ is 
\begin{equation}\label{eq:frob_cond}
    \left(\frac{6B\weber}{\pe_k}\right)\epsilon_k[\pi_k],
\end{equation}
where $\weber \in H$ is the value of a Weber function at a complex number corresponding to $E$
(see \S 2 in \cite{RS_pointCM2009} for the definition),
and $\epsilon_k$ is $\pm 1$ depending on $\pi_k^3 \bmod{4}$
(see Proposition 6.2 in \cite{RS_pointCM2009}).

For condition (ii),
Abatzoglou et al. showed the following theorem.
\begin{theorem}[Theorem 4.4 of \cite{Abatzoglou2015}]\label{thm:abat_isog}
    Let $\lambda$ be a prime ideal of $\order_K$ such that
    $\pe_k $ does not divide $\Norm_{K/\Q}(\lambda)$,
    and $\varphi: E \to E'$ an isogeny with kernel $E[\bar{\lambda}]$.
    Let $P \in E(M)$, $F = M(E'[\lambda])$, and $L = F(\widehat{\varphi}^{-1}(P))$,
    where $\widehat{\varphi}$ is the dual isogeny of $\varphi$.
    If $\pe_k$ is a prime ideal then
    the following are equivalent:
    \begin{enumerate}
        \renewcommand{\labelenumi}{{\rm (\alph{enumi})}}
        \item $\tilde{P} \not\in [\lambda] \tilde{E}(\order_{M}/\pe_k)$,
        \item $\pe_k$ splits completely in $F$ and
            $\pe_k$ does not split completely in $L$.
            \label{cond:pk_split}
    \end{enumerate}
\end{theorem}
To determine the condition on $k$ satisfying (b) of this theorem,
they restrict their attention to the case $M = F$
and this field is the Hilbert class field of $K$.
This restriction makes $\lambda$ an ideal above $2$
except in the case $K = \Q(\sqrt{-3})$.
For a detailed discussion, see \S 4 in \cite{Abatzoglou2015}.

In order to construct a primality proving algorithm
for numbers to which
the $n \pm 1$ algorithms cannot be applied,
we need to restrict $K$ to a field in which $2$ splits.
Otherwise,
$F_k - 1$ is highly factored.
The imaginary quadratic fields in which $2$ splits of class number one or two are $\Q(\sqrt{-7})$ and $\Q(\sqrt{-15})$.
A primality proving algorithm for $\Q(\sqrt{-7})$ was given by \cite{Abatzoglou_2013}
and one for $\Q(\sqrt{-15})$ by \cite{Abatzoglou2015}.

\section{Primality Proof for class number three}\label{sec:new_method}
In this section,
we give a method to construct a primality proving algorithm for imaginary quadratic fields of class number three in which $2$ splits.
In particular, we give a systematic method for constructing $\pe_k$ such that
$\pi_k - 1$ is highly factored.

\subsection{Construction of $\pe_k$}\label{subsec:method_pek}
Let $K$ be an imaginary quadratic field and $H$ the Hilbert class field of $K$.
As in the previous section,
we let $\pe_k$ be an ideal of $\order_H$,
$\pi_k$ a generator of $\Norm_{H/K}(\pe_k)$,
and $F_k = \Norm_{K/\Q}(\pi_k)$ for $k \in \N$.
In this setting, ``$\pi_k - 1$ is highly factored'' means that
$\pi_k  - 1$ is of the form $A\gamma$
satisfying $A$ is divisible only by a prime ideal of $\order_K$ above $2$, 
and $F_k > 16\Norm_{K/\Q}(\gamma)^2$.
The latter condition comes from Theorem \ref{thm:abat2}.

First, we consider a more general setting.
Let $h$ be the class number of $K$
and $\xi \in \order_H$ a primitive element of $H/K$,
i.e., $H = K(\xi)$.
We can write the minimal polynomial $f(x)$ of $\xi$ over $K$ as
\begin{equation}
    f(x) =
    x^h + c_mx^m + \cdots + c_0,\quad
    c_0,\dots,c_m \in \order_K
    \mbox{ and }
    m < h.
\end{equation}

Let $\alpha$ be an element in $\order_K$
such that $\alpha$ is divisible only by a prime ideal of
$\order_K$ above $2$.
We define $\pe_k$ as an ideal generated by $1 - \alpha^k\xi$.
Then we have
\begin{align*}
    \pi_k &= \Norm_{H/K}(1 - \alpha^k\xi)
        = \alpha^{kh}\Norm_{H/\Q}(\alpha^{-k} - \xi)
        = \alpha^{kh}f(\alpha^{-k})\\
        &= 1 + c_m\alpha^{(h-m)k} + c_{m-1}\alpha^{(h-(m-1))k} + \cdots + c_0\alpha^{hk}\\
        &= 1 + \alpha^{(h-m)k}(c_m + c_{m-1}\alpha + \cdots + c_0\alpha^{mk}).
\end{align*}
In this case,
if $m < h/2$
then the norm condition in Theorem \ref{thm:abat2} holds.
More precisely, we have the following lemma.
\begin{lemma}\label{lem:norm_condition}
    We use the above notation and assume $m < h/2$.
    Let $\gamma_k = c_m + c_{m-1}\alpha + \cdots + c_0\alpha^{mk}$.
    Then 
    we have $F_k > 16\Norm_{K/\Q}(\gamma_k)^2$ for sufficiently large $k$.
\end{lemma}
\begin{proof}
    It suffices to show that
    $\displaystyle\lim_{k\to\infty}\Norm_{K/\Q}(\gamma_k)^2/F_k = 0$.
    Since $\alpha$ is divisible by a prime ideal above $2$,
    we have $\Norm_{K/\Q}(\alpha) \geq 2$.
    Therefore,
    the limit of a complex number
    $\pi_k/\alpha^{hk} = c_0 + c_1\alpha^{-k} + \dots + c_m\alpha^{-mk} + \alpha^{-hk}$
    as $k$ approaches infinity equals $c_0$.
    On the other hand,
    the condition $m < h/2$ shows $\displaystyle\lim_{k\to\infty}\gamma_k/\alpha^{hk/2} = 0$.
    Taking the norms of these limits shows that
    $$
    \lim_{k\to\infty}F_k/\Norm_{K/\Q}(\alpha)^{hk} = c_0^2
    \mbox{ and }
    \lim_{k\to\infty}\Norm_{K/\Q}(\gamma_k)^2/\Norm_{K/\Q}(\alpha)^{hk} = 0.
    $$
    Since $c_0 \neq 0$,
    we have
    $\displaystyle\lim_{k\to\infty}\Norm_{K/\Q}(\gamma_k)^2/F_k = 0$.
\end{proof}

As we stated in \S \ref{sec:prev_work},
the framework also works in the case $\gamma$ depends on $k$.
Therefore, if we find $\xi$ with $m < h/2$
then we can apply the primality proving in \cite{Abatzoglou2015}.
For general class numbers, we do not know how to find such $\xi$.
However, in the case that the class number of $K$ is three,
we can always find such $\xi$ by taking a trace-zero element.
\begin{remark}
    We can construct a primality proving algorithm for 
    a generator of the form $1 - \beta\alpha^k\xi$ of $\pe_k$ with cofactor $\beta \in \order_K$
    in the same way as described in the following sections.
    In this paper, we describe the case $\beta = 1$ to simplify the notation.
\end{remark}

\subsection{Settings}\label{subsec:our_setting}
We use the notation in the previous subsection.
Suppose that the class number of $K$ is three and that $2$ splits in $K$.
We choose a primitive element $\xi$ of $H/K$ so that the trace of $\xi$ is zero,
i.e., the minimal polynomial of $\xi$ over $K$ is $x^3 + c_1x + c_0$.
Let $\lambda$ be a prime ideal of $\order_K$ above $2$,
and $\alpha$ a generator of $\lambda^3$.
For $k \in \N$, we define $p_k = 1 - \alpha^k\xi$ and $\pe_k$ is the ideal of $\order_H$ generated by $p_k$.
In summary, our setting is
\begin{align*}
    \pe_k &= p_k\order_H = (1 - \alpha^k\xi)\order_H,\\
    \pi_k &= \Norm_{H/K}(p_k) = 1 + c_1\alpha^{2k} + c_0\alpha^{3k},\\
    F_k &= \Norm_{K/\Q}(\pi_k).
\end{align*}

Our goal is to construct a primality proving algorithm for $F_k$ by using an elliptic curve with CM by $\order_K$.
For this, we choose $\xi$ so that $F_k$ is a prime number
if $\pe_k$ is a prime ideal
since the primality of $\pe_k$ can be determined by Theorem \ref{thm:abat2}.
The following lemma gives such a condition on $\xi$.
\begin{lemma}\label{lem:cond_xi}
    Let $t$ be the trace of $\alpha$.
    Suppose that
    $1 + c_1 + c_0$
    and $1 + c_1 + c_0t$
    are not congruent to $1$ modulo $\bar{\lambda}^3$.
    Then $\pe_k$ is a prime ideal
    if and only if $F_k$ is a prime number.
\end{lemma}

\begin{proof}
    First, we show that if $\pe_k$ is a prime ideal
    then $\pi_k\order_K$ is a prime ideal.
    Let $\kappa$ be a prime ideal dividing $\pi_k\order_K$ in $K$.
    Because the minimal polynomial $x^3 + c_1x + c_0$ of $\xi$ has a root $\alpha^{-k}$ modulo $\pi_k\order_K$,
    there exists a prime ideal above $\kappa$ in $H$ whose inertia degree over $\kappa$ is one.
    From this and the fact that $H$ is the Hilbert class field of $K$,
    it follows that $\kappa$ splits completely in $H$.
    Therefore, if $\pe_k$ is prime then the inertia degree of $\pe_k$ over $K$ is one.
    This means that $\pi_k$, the norm of $\pe_k$ over $K$, is prime in $K$.

    Next, we show that if $\pi_k\order_K$ is a prime ideal
    then $F_k$ is a prime number.
    Suppose that $\pi_k\order_K$ is a prime ideal.
    We assume that $F_k$ is not prime.
    Then $F_k$ is the square of a rational prime.
    Therefore, we have $F_k \equiv 1 \pmod{8}$ or $F_k = 2^2$.
    The latter case occurs only when $2$ is inert in $K$,
    so $\alpha$ is divisible by $2$.
    In this case, $F_k$ must be odd. This contradicts the fact that $F_k = 2^2$.
    Therefore, we conclude that $F_k \equiv 1 \pmod{8}$.
    Since $\overline{\pi_k} \equiv 1 \pmod{\bar{\lambda}^3}$,
    we have $F_k \equiv \pi_k \pmod{\bar{\lambda}^3}$.
    This implies that $\pi_k \equiv 1 \pmod{\bar{\lambda}^3}$.
    On the other hand,
    we have
    $$\pi_k \equiv 1 + c_1t^{2k} + c_0t^{3k} \pmod{\bar{\lambda}^3}$$
    since $\alpha \equiv t \pmod{\bar{\lambda}^3}$.
    Since $t$ is odd and $\bar{\lambda}^3$ divides $8$,
    it follows that $\pi_k$ is congruent to
    $1 + c_1 + c_0$
    or
    $1 + c_1 + c_0t$
    modulo $\bar{\lambda}^3$
    depending on whether $k$ is even or odd.
    This contradicts our assumption. Therefore, $F_k$ must be prime.
    
    The other direction is obvious.
\end{proof}

From now on, we restrict our attention to $\xi$ satisfying the assumption in Lemma \ref{lem:cond_xi}.
We can find such $\xi$ for any imaginary quadratic field of class number three
in which $2$ splits.
There are exactly two such quadratic fields, $\Q(\sqrt{-23})$ and $\Q(\sqrt{-31})$.
We can take $\xi$ as a root of $x^3 - x - 1$ for $\Q(\sqrt{-23})$
and a root of $x^3 + x + 1$ for $\Q(\sqrt{-31})$.

The remaining tasks are to find an elliptic curve and a point on it
and to determine the condition on $k$ satisfying Assumption \ref{assump:frob_inv}.

\subsection{Determining the condition on $k$}\label{subsec:cond_k}
We can use the method in \cite{Abatzoglou_2013,Abatzoglou2015} to determine the condition on $k$,
but the authors described methods for searching conditions on $k$ separately for each form of $F_k$.
Here, we give a slightly more unified method and explicit algorithms for this.

Let $E$ be an elliptic curve with CM by $\order_K$
defined by $y^2 = x^3 + Ax + B$ with $A, B \in \order_H$.
We can suppose that $B$ is square in $\order_H$ since
$y^2 = x^3  +Au^2x + Bu^3$ is isomorphic to $E$ for any $u \in H^\times$.
We suppose that $B$ is square
and let $\beta$ be a square root of $B$.
Then $(0, \beta)$ is a point in $E(H)$.
We let $P = (0, \beta)$.
As in \S \ref{sec:prev_work},
when $\pe_k$ is a prime ideal, we denote the reductions
modulo $\pe_k$ of $E$ and $P$ by $\tilde{E}$ and $\tilde{P}$.

The purpose of this subsection is to give a method to compute sets $T_0$, $T_1$, and $T_2$
satisfying the following conditions:
\begin{align*}
    &T_0 = \{k \in \N \mid \pe_k \mbox{ does not divide } \mathrm{disc}(E)\},\\
    &k \in T_1 \mbox{ and } \pe_k \mbox{ is a prime ideal } \Rightarrow \mbox{(i) of Assumption \ref{assump:frob_inv} holds},\\
    &k \in T_2 \mbox{ and } \pe_k \mbox{ is a prime ideal } \Rightarrow \mbox{(ii) of Assumption \ref{assump:frob_inv} holds}.
\end{align*}
We can apply Theorem \ref{thm:abat2} to $k \in T_0 \cap T_1 \cap T_2$ and
obtain a primality proving algorithm.
In the following, we show that these sets are determined by
the condition on $k$ modulo some integer.
Determining $T_0$ is trivial, so we focus on $T_1$ and $T_2$ in the following.
First, 
we show that quadratic residue symbols modulo $\pe_k$ are determined by $k$ modulo some integers.
Next, 
we show that we can define $T_1$ and $T_2$ as
sets of integers satisfying conditions on quadratic residue symbols.
Finally, we give explicit algorithms to determine $T_1$ and $T_2$

\subsubsection{Computing quadratic residue symbols}
To determine quadratic residue symbols modulo $\pe_k$,
we use the following theorem.
\begin{theorem}\label{thm:cond_k_quad}
    Fix $a \in H^\times$.
    For a prime ideal $\il$ of $\order_H$,
    we denote the order of $(\order_H/\il)^\times$ by $N_\il$.
    Then the quadratic residue symbol $\left(\frac{a}{\pe_k}\right)$ is determined by
    $k$ modulo the least common multiple of $2$ and $N_\il$
    for the prime ideals $\il$ of $\order_H$ dividing $a$ and not dividing $2$.
\end{theorem}
To prove this theorem, we need the following lemmas.
\begin{lemma}\label{lem:hilbert}
    Let $L$ be a number field in which $2$ is not ramified,
    and $\mathfrak{l}$ a prime ideal of $\order_L$ above $2$.
    We fix an element $a$ in $L^\times$.
    Then, for $b \in \order_L$ not divisible by $\mathfrak{l}$,
    the quadratic Hilbert symbol $\left(\frac{a,b}{\mathfrak{l}}\right)$
    is determined by $b$ modulo $\mathfrak{l}^3$.
\end{lemma}

\begin{proof}
    We may assume that $\ord_\mathfrak{l}(a) = 0 \mbox{ or } 1$.
    Let $\order_\il$ be the valuation ring of $L_\mathfrak{l}$.
    It suffices to show that $\left(\frac{a,b}{\mathfrak{l}}\right) = 1$
    if and only if
    $ax^2 + by^2 \equiv z^2 \pmod{\il^3}$ has a solution $(x, y, z)$ in $\order_\il^3$
    such that at least one of $x, y, z$ is in $\order_\il^\times$.
    
    It is obvious that the former implies the latter.
    Suppose that the latter holds.
    Let $(x_0, y_0, z_0)$ be a solution of the equation in the latter.  
    If $\ord_\il(y_0)$ and $\ord_\il(z_0)$ are positive
    then $\ord_\il(x_0)$ is also positive
    since $\ord_\il(a) \leq 1$.
    Therefore, $y_0$ or $z_0$ is in $\order_\il^\times$.
    Assume that $y_0 \in \order_\il^\times$.
    We define a polynomial $f(y) = y^2 + b^{-1}ax_0^2 - b^{-1}z_0^2$.
    Then $\ord_\il(f(y_0)) \geq 3$ and $\ord_\il(f'(y_0)) = \ord_\il(2) = 1$.
    Therefore, 
    we can apply Hensel's lemma \cite[Proposition 4.1.37]{cohen_nt} to $f(y)$
    and obtain a root $y_1 \in \order_\il^\times$ of $f(y)$.
    This means $\left(\frac{a,b}{\il}\right) = 1$.
    The case $z_0 \in \order_\il^\times$ can be proven
    by defining $f$ as $f(z) = z^2 - ax_0^2 - by_0^2$.
\end{proof}

\begin{lemma}\label{lem:cond_hilbert}
    Fix a prime ideal $\il$ of $\order_H$ above $2$ and $a \in H^\times$.
    The quadratic Hilbert symbol $\left(\frac{a, p_k}{\il}\right)$ is determined by $k \bmod{2}$.
\end{lemma}

\begin{proof}
    Recall that $p_k = 1 - \alpha^k\xi$, where $\alpha$ is a generator
    of $\lambda^3$
    and that $\lambda$ is a prime ideal of $\order_K$ above $2$.

    From Lemma \ref{lem:hilbert},
    the quadratic Hilbert symbol in this lemma
    is determined by $p_k \bmod{\il^3}$.
    If $\il$ divides $\lambda$ then $p_k \equiv 1 \pmod{\il^3}$ for any $k$.
    Otherwise, $\il$ divides $\bar{\lambda}$,
    so $\alpha \equiv \mathrm{Tr}_{K/\Q}(\alpha) \pmod{\il^3}$.
    Since $\mathrm{Tr}_{K/\Q}(\alpha)$ is an odd integer,
    $\alpha^k \bmod{\il^3}$ is determined by $k \bmod{2}$.
\end{proof}

\begin{proof}[Proof of Theorem \ref{thm:cond_k_quad}]
    By the product formula of Hilbert symbols \cite[Theorem (8.1)]{neukirch2013algebraic},
    we have
    \begin{equation}\label{eq:quadratic_pek}
        \left(\frac{a}{\pe_k}\right)
            = \prod_{\mathfrak{l} \mid 2}\left(\frac{a, p_k}{\mathfrak{l}}\right)
                \prod_{\mathfrak{l} \mid a\order_H \mbox{ and } \mathfrak{l} \nmid 2}\left(\frac{p_k}{\mathfrak{l}}\right)^{\ord_{l}(a)}
    \end{equation}
    for $a \in H^\times$, where $\mathfrak{l}$ runs over the prime ideals of $H$ different from $\pe_k$.
    From Lemma \ref{lem:cond_hilbert},
    the first product of the right-hand side of \eqref{eq:quadratic_pek}
    is determined by $k \bmod{2}$.
    Since a residue symbol modulo $\il$ is determined by $k \bmod{N_\il}$,
    the second product of the right-hand side of \eqref{eq:quadratic_pek}
    is determined by $k$ modulo the least common multiple of $N_\il$
    for the prime ideals $\il$ of $\order_H$ dividing $a$ and not dividing $2$.
    This completes the proof.
\end{proof}

We denote 
the right-hand side of \eqref{eq:quadratic_pek}
by $s_k(a)$.
Note that $s_k(a)$ is defined even if $\pe_k$ is not a prime ideal
and equal to $\left(\frac{a}{\pe_k}\right)$ if $\pe_k$ is a prime ideal.

\subsubsection{Determining $T_1$}
As mentioned in \S \ref{sec:prev_work},
we can use Theorem 5.3 of \cite{RS} for determining $T_1$.
Since $B$ is square in $\order_H$,
if $\pe_k$ is a prime ideal and
$\left(\frac{6\weber}{\pe_k}\right)\epsilon_k = 1$
then (i) in Assumption \ref{assump:frob_inv} holds.
Therefore, we can define
\begin{equation}
    T_1 \coloneqq \{k \in \N \mid s_k(6\weber)\epsilon_k = 1\}.
\end{equation}
In the following,
we show how to compute $\epsilon_k$ and $\weber$.

First, we show that $\epsilon_k$ is determined by $k \bmod{2}$.
Let $D$ be the discriminant of $\order_K$.
We fix a square root of $D$ in $K$ and denote it by $\sqrt{D}$.
Proposition 6.2 of \cite{RS} says that, in our case,
\begin{align}\label{eq:epsilon}
    \epsilon_k =
    \begin{cases}
        1 & \mbox{ if } \pi_k^3 \equiv 1 \mbox{ or } -\sqrt{D} \pmod{4},\\
        -1 & \mbox{ otherwise}.
    \end{cases}
\end{align}
Since 2 splits in $K$,
we have $(\order_K/4\order_K)^\times \cong (\Z/4\Z)^\times \oplus (\Z/4\Z)^\times$.
Therefore, $\pi_k \bmod{4}$ is determined by $k \bmod{2}$,
so $\epsilon_k$ is also determined by $k \bmod{2}$.

Next, we show how to compute $\weber$.
We have $\weber \in H$ (Lemma 3.4 of \cite{RS}) and
\begin{equation}
    \weber^2 = j(E) - 1728.
\end{equation}
This equation determines $\weber$ up to sign.
The sign can be determined as follows.
Let $p$ be a rational prime
congruent to $3$ modulo $4$,
splitting in $\order_K$,
and not dividing the discriminant of $E$.
Let $\pi$ be a prime of $\order_K$ above $p$
and let $\pe$ be a a prime ideal of $\order_H$ above $\pi$.
We denote the reduction of $E$ modulo $\pe$ by $\bar{E}$.
Theorem 5.3 of \cite{RS} says that
the Frobenius endomorphism of $\bar{E}$ is
$\left(\frac{6\weber}{\pe}\right)\epsilon(\pi)[\pi]$,
where $\epsilon(\pi)$ is defined in the same way as in \eqref{eq:epsilon}.
Since $p \equiv 3 \pmod{4}$ and $p$ splits in $\order_K$,
we have $\left(\frac{-1}{\pe}\right) = -1$.
Therefore, we can determine the sign of $\weber$ by checking whether
\begin{equation}
    \left(\frac{6\weber}{\pe}\right)\epsilon(\pi)[\pi]Q = Q
\end{equation}
for $Q \in \bar{E}(\order_H/\pe)$ such that $[2]Q \neq O_{\bar{E}}$.
More precisely, the sign of $\weber$ can be determined as follows.
\begin{enumerate}
    \item Compute a square root $t$ of $j(E) - 1728$.
    \item Determine $\epsilon(\pi)$ by \eqref{eq:epsilon}.
    \item Take $Q \in \bar{E}(\order_H/\pe)$ such that $[2]Q \neq O_{\bar{E}}$.
    \item If $\left[\left(\frac{6t}{\pe}\right)\epsilon(\pi)\pe\right]Q = Q$
        then $\weber = t$, otherwise $\weber = -t$.
\end{enumerate}
We can take $p = 59$ in the case $K = \Q(\sqrt{-23})$
and $p = 47$ in the case $K = \Q(\sqrt{-31})$.

\subsubsection{Determining $T_2$}
We use Theorem \ref{thm:abat_isog}
to determine $T_2$.
In particular, we show the following theorem.

\begin{theorem}\label{thm:T2}
    There exists a polynomial $f$ with coefficients in $\order_H$ of degree $2$
    such that
    when $\pe_k$ is a prime ideal, the following are equivalent:
    \begin{enumerate}
        \renewcommand{\labelenumi}{{\rm (\alph{enumi})}}
        \item $\tilde{P} \not\in [\lambda]\tilde{E}(\order_H/\pe_k)$,
        \item $f \bmod \pe_k$ has a multiple root
                or does not have a root in $\order_H/\pe_k$.
    \end{enumerate}
\end{theorem}

\begin{proof}
    Let $\varphi$ be an isogeny from $E$ with kernel $E[\bar{\lambda}]$
    and $E': y^2 = x^3 + A'x + B'$ its codomain.
    The theory of complex multiplication shows that
    $E'$ has also CM by $\order_K$.
    Therefore, we can assume that $A'$ and $B'$ are in $\order_H$.

    From Lemma 4.6 of \cite{Abatzoglou2015},
    we have $E'[\lambda] \subset E'(H)$.
    Therefore, the field $F$ in Theorem \ref{thm:abat_isog} is equal to $H$ in our case.
    Consequently,
    when $\pe_k$ is a prime ideal,
    (a) holds if and only if
    $\pe_k$ does not split completely in $H(\widehat{\varphi}^{-1}(P))$.

    Let $x_0$ be the $x$-coordinate of the generator of $E'[\lambda]$.
    Since $\lambda$ divides $2$, we have $x_0$ is a root of $x^3 + A'x + B'$.
    Therefore, we have $x_0 \in \order_H$.

    By V\'elu's formula \cite{Velu1971},
    the image of an affine point $(x, y) \in E'$ under $\widehat{\varphi}$ is given by
    \begin{equation}\label{eq:isogeny}
        \widehat{\varphi}(x, y) = \left(\frac{f(x)}{x - x_0}, \frac{g(x)}{(x - x_0)^2}y\right),
    \end{equation}
    where and $f, g$ are univariate polynomials over $H$ of degree $2$.
    From the explicit form of V\'elu's formula,
    we can show that the coefficients of $f$ and $g$ are in $\Z[A', B', x_0]$.
    Therefore, the coefficients of $f$ and $g$ are in $\order_H$.

    Let $H_f$ be the splitting field of $f$ over $H$.
    Then $H_f$ is a subfield of $H(\widehat{\varphi}^{-1}(P))$.
    Note that the degrees of these fields over $H$ are at most $2$
    and that $H_f \subseteq H(\widehat{\varphi}^{-1}(P))$.
    Therefore, we have $H_f = H(\widehat{\varphi}^{-1}(P))$ if $H_f \neq H$.

    Assume that $H_f = H$.
    Let $Q$ be a point in $\widehat{\varphi}^{-1}(P)$.
    The $x$-coordinate of $Q$ is a root of $f$, so it is in $H$.
    Then $Q^{\sigma} = \pm Q$ for any $\sigma$ in the Galois group of $H(\widehat{\varphi}^{-1}(P))/H$.
    If there exists $\sigma$ such that $Q^\sigma = -Q$
    then we have $-P = \widehat{\varphi}(-Q) = \widehat{\varphi}(Q^\sigma) = (\widehat{\varphi}(Q))^\sigma = P^\sigma = P$.
    This means that $P \in E[2]$.
    This contradicts $P = (0, \beta)$ with $\beta \neq 0$.
    Note that $B \neq 0$ since $E$ does not have CM by the Gauss integer ring.
    Therefore, we have $Q \in E'(H)$.
    This means that $H(\widehat{\varphi}^{-1}(P)) = H$.
    
    Therefore, we have $H_f = H(\widehat{\varphi}^{-1}(P))$ in any case.
    This means that when $\pe_k$ is a prime ideal,
    it follows that $\pe_k$ does not split completely in $H(\widehat{\varphi}^{-1}(P))$
    if and only if (b) holds.
\end{proof}

Let $f$ be the same as in the proof of Theorem \ref{thm:T2}
and denote the discriminant of $f$ by $\mathrm{disc}(f)$.
Then we can define
\begin{equation}
    T_2 \coloneqq
    \{k \in \N \mid s_k(\mathrm{disc}(f)) \neq 1 \}.
\end{equation}

\begin{remark}
    For computing the isogeny $\varphi$ with kernel $E[\bar{\lambda}]$,
    we need to compute the action of $\bar{\lambda}$ on $E$.
    It suffices for this to obtain $[\tau]$ as a rational map,
    where $\tau$ is an element of $\order_K$ such that $\order_K = \Z + \Z\tau$.
    In our case, the norm of $\tau$ is small ($6$ or $8$ in the case $K = \Q(\sqrt{-23})$ or $\Q(\sqrt{-31})$, respectively).
    Therefore, the rational map of $[\tau]$ consists of polynomials of degree at most $8$.
    We can find the rational map of $[\tau]$ by computing all the cyclic isogenies from $E$ of degree $\Norm_{K/\Q}(\tau)$ by V\'elu's formula.
    There are exactly two such isogenies from $E$ to $E$, which correspond to $[\tau]$ and $[\bar{\tau}]$.
    We can determine the desired $[\tau]$ by computing these invariant differentials.
\end{remark}

\subsubsection{Algorithm}
The remaining task for determining $T_1$ and $T_2$ explicitly
is to determine the condition on $k$ to satisfy $s_k(a) \in S$
for given $a \in \order_H$ and $S \subset \{-1, 0, 1\}$.
Theorem \ref{thm:cond_k_quad} shows that
there exist a positive integer $M$ and a subset $\mathcal{T}_{a,S}$ of $\Z/M\Z$
such that
$s_k(a) \in S$
if and only if 
$k \bmod{M} \in \mathcal{T}_{a,S}$.

We can determine $M$ and a subset $\mathcal{T}_{a,S}$ by the following algorithm.
\begin{alg}\label{alg:conditon_k}
    For $a \in \order_H$ and $S \subset \{-1, 0, 1\}$,
    return
    a positive integer $M$ and a subset $\mathcal{T}_{a,S}$ of $\Z/M\Z$
    such that
    $s_k(a) \in S$
    if and only if 
    $k \bmod{M} \in \mathcal{T}_{a,S}$.
    \begin{enumerate}
        \renewcommand{\labelenumi}{{\rm \arabic{enumi}.}}
        \item Set $N$ be the least common multiple of $2$ and
            $N_\mathfrak{l}$ for the prime ideals $\mathfrak{l}$ of $\order_H$
            dividing $a$ and not dividing $2$.
        \item Set $\mathcal{T} = \emptyset$.
        \item For $k \in \{1, \dots, N\}$, if $s_k(a) \in S$ then append $k$ to $\mathcal{T}$. \label{step:sk}
        \item Set $M = 
            \min\{t \in \N \mid
            t \mbox{ divides } N \mbox{ and }
            k \in T \Rightarrow k + t \in T \mbox{ for all } k \in \{1, \dots, N - t\}\}$.
        \item Return $M$ and $\{k \in \mathcal{T} \mid k \leq M\}$.
    \end{enumerate}
\end{alg}

For $T_1$, we use Algorithm \ref{alg:conditon_k} with input $6\weber$ and $\{1\}$
replacing
$s_k(a) \in S$ in Step \ref{step:sk} with $s_k(a)\epsilon_k$.
For $T_2$, the input is $\mathrm{disc}(f)$ and $\{-1, 0\}$.

\subsection{Primality proving}\label{subsec:algorithm}
Now, we describe our primality proving algorithm.
In our algorithm, 
we need to compute a square root of $D$ modulo $F_k$.
For this, we use Algorithm 2.3.8 in \cite{crandall2005prime}.
This algorithm terminates in deterministic polynomial time in $k$ in our case since $F_k \not\equiv 1 \pmod{8}$.
We note that this algorithm
can be defined even if $F_k$ is not prime.
In this case, the output is not necessarily a square of $D$ modulo $F_k$.
Therefore, if the output is not a square of $D$ modulo $F_k$
then we can determine that $F_k$ is not prime.

Let $C_k$ be the odd part of $\Norm_{K/\Q}(c_1 + c_0\alpha^k)$.
We assume that $2^{6k + 1}$ does not divide $\Norm_{K/\Q}(c_1 + c_0\alpha^k)$.
Then we can construct a primality proving
by replacing $\gamma$ in Theorem \ref{thm:abat2} with $C_k$.
In particular,
we obtain the following primality proving algorithm.
\begin{alg}\label{alg:main}
    For $k \in T_0 \cap T_1 \cap T_2$ such that $F_k > 16\Norm_{K/\Q}(c_1 + c_0\alpha^k)^2$
    and $2^{6k + 1}$ does not divide $\Norm_{K/\Q}(c_1 + c_0\alpha)$,
    this algorithm returns whether $F_k$ is prime.
    \begin{enumerate}
        \renewcommand{\labelenumi}{{\rm \arabic{enumi}.}}
        \item Compute $r$ by Algorithm 2.3.8 in \cite{crandall2005prime} with input $F_k$ and $D$.
        \item If $r^2 \not\equiv D \pmod{F_k}$ then return \textsf{False}.
        \item Let $a$ be an integer modulo $F_k$ by replacing $\sqrt{D}$ in $\alpha$ with $r$. \label{step:a_mod}
        \item Let $t$ be $a^{-k} \bmod{F_k}$.
        \item Compute $t^3 + c_1t + c_0 \bmod{F_k}$,
            where $\sqrt{D}$ in $c_0$ and $c_1$ is replaced with $r$.
            If it is not congruent to zero then replace $r$ with $-r$ and go to Step \ref{step:a_mod},
        \item Define an elliptic curve $E'$ over $\Z/F_k\Z$ and its point $P'$
            by replacing $\sqrt{D}$ with $r$ and $\xi$ with $t$ in $E$ and $P$.
        \item Compute $Q = [2^{6k-1}C_k]P'$.
        \item If $Q$ is not strongly nonzero modulo $F_k$ then return \textsf{False}.
        \item Return whether $[2]Q \equiv O_{E'} \pmod{F_k}$.
    \end{enumerate}
\end{alg}
This algorithm is analogous to Algorithm 6.3 in \cite{Abatzoglou2015},
which is a primality proving algorithm for an integer sequence obtained from $\Q(\sqrt{-15})$.
As in Algorithm 6.3 in \cite{Abatzoglou2015},
the complexity of this algorithm is $O(k^2\log k\log\log k)$.
We show the correctness of Algorithm \ref{alg:main}.
The proof is similar to that of Theorem \ref{thm:abat2} by \cite{Abatzoglou2015}.
For the proof,
we recall a lemma from \cite{Abatzoglou2015}.
\begin{lemma}[Lemma 2.5 of \cite{Abatzoglou2015}]\label{lem:CF}
    IF $C, F \in \mathbb{R}, C \geq 1$, and $F > 16C^2$,
    then $C(F^{1/4} + 1)^2 < (\sqrt{F} - 1)^2$.
\end{lemma}
\begin{theorem}\label{thm:correctnes}
    Algorithm \ref{alg:main} returns \textsf{True}
    if and only if $F_k$ is prime.
\end{theorem}

\begin{proof}
    Assume that $F_k$ is prime.
    Then $\pe_k$ is a prime ideal
    and we can define the reduction $\tilde{E}$ and $\tilde{P}$.
    From Lemma \ref{lem:cond_xi},
    we have $\Z/F_k\Z \cong \order_H/\pe_k$.
    Therefore, there exists a square root of $D$ modulo $F_k$
    and the condition in Step 2 holds.
    In Step 5, $r$ and $t$ satisfy that
    \begin{equation*}
        r \equiv \sqrt{D}
        \mbox{ and }
        t \equiv \xi \pmod{\pe_k}.
    \end{equation*}
    Therefore,
    through the isomorphism $\order_H/\pe_k \cong \Z/F_k\Z$,
    we can identify $E'$ with $\tilde{E}$ and $P'$ with $\tilde{P}$.

    Since the Frobenius endomorphism on $\tilde{E}$
    is $[\pi_k]$,
    we have
    \begin{align*}
        \tilde{E}(\order_H/\pe_k) &\cong \order_K/(\pi_k - 1)\order_K\\
            &= \order_K/(\alpha^{2k}(c_1 + c_0\alpha^k))\order_K\\
            &\cong \Z/2^{6k}\Z \oplus \order_K/(c_1 + c_0\alpha^k)\order_K.
    \end{align*}
    From this and the assumption that $2^{6k+1}$ does not divide $\Norm_{K/\Q}(c_1 + c_0\alpha^k)$,
    we have
    $[2^{6k}C_k]\tilde{P} = O_{\tilde{E}}$.
    In addition, $[2^{6k-1}C_k]\tilde{P} \neq O_{\tilde{E}}$
    since $\tilde{P} \notin [2]\tilde{E}(\order_H/\pe_K)$.
    This means that the condition in Step 8 does not hold
    and that the congruence in Step 9 holds.
    
    Conversely, we assume that Algorithm \ref{alg:main} returns \textsf{True}.
    This means that,
    for any prime factor $p$ of $F_k$,
    there exists a point
    in $E' \bmod{p}$ of order at least $2^{6k}$.

    Assume that $F_k$ is not prime.
    Let $p$ be the smallest prime factor of $F_k$.
    Then $\sqrt{F_k} \geq p$.
    By the Hasse's theorem \cite[Theorem V.1.1]{silverman2009arithmetic},
    the order of $(E' \bmod{p})(\Z/p\Z)$ is at most $(\sqrt{p} + 1)^2$.
    Therefore, we have
    \begin{align*}
        (F_k^{1/4} + 1)^2 &\geq (\sqrt{p} + 1)^2
            \geq 2^{6k} = \Norm_{K/\Q}\left(\frac{\pi_k - 1}{c_1 + c_0\alpha^k}\right)\\
            &= \frac{(\pi_k - 1)(\overline{\pi_k} - 1)}{\Norm_{K/\Q}(c_1 + c_0\alpha^k)}
            = \frac{F_k - \mathrm{Tr}_{K/\Q}(\pi_k) + 1}{\Norm_{K/\Q}(c_1 + c_0\alpha^k)}\\
            &\geq \frac{F_k - 2\sqrt{F_k} + 1}{\Norm_{K/\Q}(c_1 + c_0\alpha^k)}
            = \frac{(\sqrt{F_k} - 1)^2}{\Norm_{K/\Q}(c_1 + c_0\alpha^k)}.
    \end{align*}
    The final inequality comes from the fact that
    $\trace_{K/\Q}(\beta)^2 \leq 4\Norm_{K/\Q}(\beta)$ for any $\beta \in K$.
    From the above inequality, we have
    \begin{equation*}
        \Norm_{K/\Q}(c_1 + c_0\alpha^k)(F_k^{1/4} + 1)^2
        \geq (\sqrt{F_k} - 1)^2.
    \end{equation*}
    This contradicts Lemma \ref{lem:CF} with $C = \Norm_{K/\Q}(c_1 + c_0\alpha^k)$
    and $F = F_k$.
    Therefore, $F_k$ must be prime.
\end{proof}

\section{Examples}\label{sec:examples}
As mentioned above,
our algorithm can be applied to the following two cases.
\begin{enumerate}
    \item $K = \Q(\sqrt{-23})$ and $\xi$ is a root of $x^3 - x - 1$,
    \item $K = \Q(\sqrt{-31})$ and $\xi$ is a root of $x^3 + x + 1$.
\end{enumerate}
In these cases, the sequences $\{F_k\}$ determined the primality are
\begin{enumerate}
    \item $F_k = \Norm_{\Q(\sqrt{-23})/\Q}
        \left(1 - \left(\frac{3 + \sqrt{-23}}{2}\right)^{2k}
            - \left(\frac{3 + \sqrt{-23}}{2}\right)^{3k}
        \right)$,
    \item $F_k = \Norm_{\Q(\sqrt{-31})/\Q}
        \left(1 + \left(\frac{1 + \sqrt{-31}}{2}\right)^{2k}
            + \left(\frac{1 + \sqrt{-31}}{2}\right)^{3k}
        \right)$.
\end{enumerate}

In these cases,
the conditions on $k$ in Algorithm \ref{alg:main}
hold for $k \geq 2$.
In particular, 
we have the following lemmas.
\begin{lemma}\label{lem:Fk16N}
    Let $k \geq 2$ be an integer
    and $F_k$ one of the two integers defined above.
    We define
    \begin{itemize}
        \item $c_0 = c_1 = -1$ and $\alpha = \frac{3 + \sqrt{-23}}{2}$
            if $F_k$ is the sequence defined by (1),
        \item $c_0 = c_1 = 1$ and $\alpha = \frac{1 + \sqrt{-31}}{2}$
            if $F_k$ is the sequence defined by (2).
    \end{itemize}
    Then we have 
    \begin{equation}\label{eq:cond_norm}
        F_k > 16\Norm_{K/\Q}(c_1 + c_0\alpha^k)^2.
    \end{equation}
\end{lemma}

\begin{proof}
    Let $K$ be $\Q(\sqrt{-23})$ in the case (1) or $\Q(\sqrt{-31})$ in the case (2).
    Note that we have $F_k = \Norm_{K/\Q}(1 + c_1\alpha^{2k} + c_0\alpha^{3k})$ in the both cases.

    For any $\beta \in K$, we have 
    $\trace_{K/\Q}(\beta)^2 \leq 4\Norm_{K/\Q}(\beta)$.
    Since
    $c_0, c_1 \in \{\pm 1\}$ and $\Norm_{K/\Q}(\alpha) = 8$, we have
    \begin{align*}
        F_k &= 8^{3k} + 8^{2k} + 1
            + c_1 \trace_{K/\Q}(\alpha^{2k}) + c_0 \trace_{K/\Q}(\alpha^{3k})
            + c_0c_1 8^{2k}\trace_{K/\Q}(\alpha^{k})\\
            &\geq 8^{3k} + 8^{2k} + 1
                - 4(8^{k} + 8^{1.5k} + 8^{2.5k})\\
            &> 8^{3k} - 8\cdot8^{2.5k}.
    \end{align*}
    On the other hand,
    \begin{align*}
        \Norm_{K/\Q}(c_1 + c_0\alpha^k) = \Norm_{K/\Q}(1 + \alpha^k) = 8^k + 1 + \trace_{K/\Q}(\alpha^k) < 2\cdot8^k.
    \end{align*}
    Therefore, \eqref{eq:cond_norm} holds if
    $8^{3k} - 8\cdot8^{2.5k} > 16(2\cdot8^{k})^2$.
    If $k \geq 5$ then
    this inequality holds since $8^{3k} > 72\cdot8^{2.5k}$.
    Computing explicitly both sides in \eqref{eq:cond_norm} for $k = 2, 3, 4$,
    we obtain our assertion.
\end{proof}

\begin{lemma}\label{lam:6k}
    Let $K$ be an imaginary quadratic field,
    $k$ be a positive integer, $c_0, c_1 \in \{\pm 1\}$.
    Let $\alpha \in \order_K$ such that $\Norm_{K/\Q}(\alpha) = 8$.
    Then $2^{6k}$ does not divide $N_{K/\Q}(c_0 + c_1\alpha^k)$.
\end{lemma}
\begin{proof}
    It suffices to show $N_{K/\Q}(c_0 + c_1\alpha^k) < 2^{6k}$.
    As in the proof of Lemma \ref{lem:Fk16N},
    we have $\trace_{K/\Q}(\alpha^k) \leq 2\sqrt{\Norm_{K/\Q}(\alpha^k)} = 2\cdot2^{1.5k}$.
    Therefore,
    \begin{equation*}
        N_{K/\Q}(c_0 + c_1\alpha^k) = 1 + 2^{3k} \pm\trace_{K/\Q}(\alpha^k)
            \leq (1 + 2^{1.5k})^2
            < (2\cdot2^{1.5k})^2 < 2^{6k}.
    \end{equation*}
\end{proof}

We implemented the primality proving algorithm in this paper for these sequences.
We used SageMath \cite{sagemath} for implementation and our code is available at
\url{https://github.com/hiroshi-onuki/ecpp_cm_h3}.

As in \S \ref{subsec:cond_k},
we can determine the sets $T_0$, $T_1$, and $T_2$ explicitly.
In particular, $T_0$ is determined by the discriminant of the elliptic curve $E$,
and $T_1$ and $T_2$ are determined by Algorithm \ref{alg:conditon_k}.
The set $T_0 \cap T_1 \cap T_2$ is determined modulo $1320$ in case (1),
and $7920$ in case (2)
for $E$'s that we used
(see \textsf{n23\_def\_symbols.sage} and \textsf{n31\_def\_symbols.sage} of our code
for the explicit construction).

We ran our primality proving algorithm 
for $k$ in $[2, 20000] \cap T_0 \cap T_1 \cap T_2$.
Our computation showed the following fact.

\begin{theorem}
    Let $k \in [2, 20000] \cap T_0 \cap T_1 \cap T_2$.
    \begin{itemize}
        \item $F_k$ defined by (1) is prime if and only if $k = 100, 120$, and $313$,
        \item $F_k$ defined by (2) is prime if and only if $k = 23, 191, 1350$, and $2071$.
    \end{itemize}
\end{theorem}
The computations were performed on a single core of an Apple M1 CPU
and took about one week in both cases. 

\subsection*{Acknowledgements}
The author would like to thank the anonymous reviewers for their useful feedback.
This research was in part conducted under a contract of
``Research and development on new generation cryptography for secure wireless communication services''
among
``Research and Development for Expansion of Radio Wave Resources (JPJ000254)'',
which was supported by the Ministry of Internal Affairs and Communications, Japan.

\bibliographystyle{amsplain}
\bibliography{elliptic_curve,isog}

\end{document}